\documentclass[12pt]{amsart}
\usepackage{amsthm,amssymb}
\usepackage{hyperref}
\usepackage{epsfig}
\usepackage{pstricks}
\usepackage{pst-plot}
\usepackage{tabmac}

\usepackage[margin=1.1in]{geometry}

\newtheorem{theorem}{Theorem}[section]
\newtheorem{thm}[theorem]{Theorem}
\newtheorem{lemma}[theorem]{Lemma}
\newtheorem{lem}[theorem]{Lemma}

\newtheorem{corollary}[theorem]{Corollary}
\newtheorem{cor}[theorem]{Corollary}

\theoremstyle{remark}
\newtheorem{example}{Example}
\newtheorem{remark}{Remark}

\def\ep{\varepsilon}
\def\ph{\varphi}
\def\Z{{\mathbb Z}}
\def\det{{\mathrm{det}}}
\def\LSym{\mathrm{LSym}}

\def\C{{\mathbb C}}

\def\D{{\overline D}}

\def\e{{\tilde e}}

\def\k{{ \kappa}}

\def\uqsln{{U_q'({\mathfrak {\hat {sl_n}}})}}

\newcommand\remind[1]{{\bf ** #1 **}}

\title{}
 \author{Thomas Lam}\address
 {Department of Mathematics\\ University of Michigan\\ Ann Arbor\\ MI 48109 USA.}
 \date{\today}
 \email{tfylam@umich.edu}
 \urladdr{http://www.math.lsa.umich.edu/\~{ }tfylam}
 \thanks{T.L. was supported by NSF grant DMS-0652641 and DMS-0901111, and by a Sloan Fellowship.}
 \author{Pavlo Pylyavskyy}\address
{Department of Mathematics\\ University of Michigan\\ Ann Arbor\\ MI 48109 USA.}
 \email{pavlo@umich.edu}
 \urladdr{http://sites.google.com/site/pylyavskyy/}
 \thanks{P.P. was supported by NSF grant DMS-0757165.}
\begin{document}
\begin{abstract}
We study products of the affine geometric crystal of type $A$ corresponding to symmetric powers of the standard representation.  The quotient of this product by the $R$-matrix action is constructed inside the unipotent loop group.  This quotient crystal has a semi-infinite limit, where the crystal structure is described in terms of limit ratios previously appearing in the study of total positivity of loop groups.
%
\end{abstract}
\title{Affine geometric crystals in unipotent loop groups}
\maketitle

\section{Introduction}
{\it Geometric crystals} were invented by Berenstein and Kazhdan \cite{BK,BK2} as birational analogues of Kashiwara's crystal graphs.  Suppose $X$ is a geometric crystal.  Then the product $X^{m}$ is also a geometric crystal, and in certain cases \cite{KNO} one has a (birational) $R$-matrix $R: X \times X \to X \times X$ giving rise to a birational action of the symmetric group $S_m$ on $X^{m}$.  Since the $R$-matrix is an isomorphism of geometric crystals, the crystal structure of $X^{m}$ can be considered as $S_m$-invariants of $X^{m}$.  Our investigations began with expressing the crystal structure in terms of the invariant rational functions $\C(X^{m})^{S_m}$ in the special case that $X = X_M$ is the {\it basic geometric crystal} corresponding to symmetric powers of the standard representation of $\uqsln$.  Equivalently, we construct the {\it quotient crystal} of $X_M^m$ by $S_m$.  It turns out that this quotient crystal can be constructed inside the unipotent loop group $U$ in an extremely natural way.

The {\it basic geometric crystal} $X_M$ of type $A_{n-1}^{(1)}$ is the variety $$X_M = \{x=(x^{(1)},x^{(2)},\ldots,x^{(n)}) \mid x^{(i)} \in \C^{*}\}=\C^n,$$ equipped with a distinguished collection of rational functions $\ep_i(x) = x^{(i+1)}$, $\ph_i(x) = x^{(i)}$, $\gamma_i(x) = x^{(i)}/x^{(i+1)}$ and a collection of rational $\C^*$-actions $e_i^c:  X_M \to X_M$, $c \in \C^*$ satisfying certain relations (see Section \ref{sec:geom}).  Our $X_M$ is essentially the geometric crystal ${\mathcal B}_L(A_{n-1}^{(1)})$ of Kashiwara, Nakashima, and Okado \cite[Section 5.2]{KNO}, with the dependence on $L$ removed.  It is a geometric analogue \cite{KNO1} of a limit of perfect (combinatorial) crystals, the latter playing an important role in the theory of vertex models.

Now consider the product $X= X_1 \times X_2 \times \cdots \times X_m$, where each $X_i \simeq X_{M}$.  In this case, the birational $R$-matrix $R_i: X_i \times X_{i+1} \to X_{i+1} \times X_{i}$ has been explicitly calculated \cite{Y}.  This same rational transformation appeared in the study of total positivity in loop groups.  Let $G= GL_n(\C((t)))$ denote the formal loop group, and let $U\subset G$ denote the maximal unipotent subgroup of $G$.  In \cite{LP}, we defined certain elements $M(x^{(1)},\ldots,x^{(n)}) \in U$ (see Section \ref{sec:whirlchev}), called {\it whirls}, depending on $n$ parameters $x^{(1)},\ldots,x^{(n)}$.  We showed that any totally nonnegative element $g \in U_{\geq 0}$ whose matrix entries were polynomials, was a product of such whirls.  It turns out that generically there is a unique non-trivial rational transformation of parameters $(x,y) \mapsto (u,w)$ such that $$M(x^{(1)},x^{(2)},\ldots,x^{(n)}) M(y^{(1)},y^{(2)},\ldots,y^{(n)}) = M(u^{(1)},u^{(2)},\ldots,u^{(n)}) M(w^{(1)},w^{(2)},\ldots,w^{(n)}).$$ Up to a shift of indices,  this transformation coincides with the birational $R$-matrices $R_i$ described above. 

Let $U^{\leq m} \subset U$ denote the closure of the set of elements $\{g = \prod_{i=1}^m M(x_i)\}$ which can be expressed as the product of $m$ whirls.  We show that $\C(U^{\leq m}) = \C(X_M^m)^{S_m}$ and that the geometric crystal structure of $X_M^m$ descends to $U^{\leq m}$, where it can be described explicitly in terms of left and right multiplication by one-parameter subgroups of $U$ (Theorems \ref{thm:Ucrystal} and \ref{thm:einv} and Corollary \ref{cor:same}).  This is reminiscent of the construction \cite{BK2} of a geometric crystal from a unipotent bicrystal, though $U^{\leq m}$ is not closed under multiplication by $U$.

One intriguing observation we make is that the geometric crystal $U^{\leq m}$ has a limit as $m \to \infty$, which morally one may think of as a quotient of the semi-infinite product $X_M^{\infty}$.  It gives rise to a ``geometric crystal'' structure on the whole unipotent loop group $U$, where $\ep_k, \ph_k$ are no longer rational functions, but are asymptotic {\it limit ratios} of the matrix coefficients of $U$ (Theorem \ref{thm:infinity}).  These limit ratios played a critical role in the factorization of totally nonnegative elements \cite{LP}.  It points to a deeper connection which we have yet to understand.


The matrix coefficients of $U^{\leq m}$ give a natural set of (algebraically independent) generators of $\C(X_M^m)^{S_m}$.  These matrix coefficients are the {\it loop elementary symmetric functions} $e_r^{(s)}(x_1,\ldots,x_m)$.  The invariants also contain distinguished elements $s_\lambda^{(s)}(x_1,\ldots,x_m)$, called the {\it loop Schur functions}.  It was observed in \cite{LP2} that the (birational) intrinsic energy function of $X_M^m$ can be expressed as a loop Schur function of dilated staircase shape. In particular, energy is a polynomial, unlike the rational functions $\ep_k, \ph_k$.  We explicitly describe the crystal operator action $e_k^c$ on loop Schur functions (Theorem \ref{thm:schuraction}).

\section{Products of geometric crystals}


\subsection{Geometric crystals}\label{sec:geom}
We shall use \cite{KNO} as our main reference for affine geometric crystals.
In this paper we shall consider affine geometric crystals of type $A$.  Fix $n > 1$.  Let $A= (a_{ij})_{i,j \in \Z/n\Z}$ denote the $A_{n-1}^{(1)}$ Cartan matrix.  Thus if $n > 2$ then $a_{ii} = 2$, $a_{ij} = -1$ for $|i-j|=1$, and $a_{ij}=0$ for $|i-j| > 1$.  For $n = 2$, we have $a_{11} = a_{22}=2$ and $a_{12}=a_{21} = -2$.  

Let $(X,\{e_i\}_{i \in \Z/n\Z},\{\ep_i\}_{i\in\Z/n\Z},\{\gamma_i\}_{i \in \Z/n\Z})$ be an affine geometric crystal for $A_{n-1}^{(1)}$.  Thus, $X$ is a complex algebraic variety, $\ep_i: X \to \C$ and $\gamma_i: X \to \C$ are rational functions, and $e_i: \C^{*} \times X \to X$ $((c,x) \mapsto e_i^c(x))$ is a rational $\C^{*}$-action, satisfying:
\begin{enumerate}
\item
The domain of $e_i^1: X \to X$ is dense in $X$ for any $i \in \Z/n\Z$.
\item
$\gamma_j(e_i^c(x)) = c^{a_{ij}} \gamma_j(x)$ for any $i,j \in \Z/n\Z$.
\item
$\ep_i(e_i^c(x)) = c^{-1}\ep_i(x)$.
\item 
for $i \neq j$ such that $a_{ij} = 0$ we have
$e_i^c e_j^{c'} = e_j^{c'} e_i^{c}$.
\item
for $i \neq j$ such that $a_{ij} = -1$ we have
$e_{i}^c e_{j}^{cc'} e_{i}^{c'} = e_j^{c'} e_i^{cc'} e_{j}^c$.
\end{enumerate}
We often abuse notation by just writing $X$ for the geometric crystal.  We define $\ph_i = \gamma_i \ep_i$, and sometimes define a geometric cyrstal by specifying $\ph_i$ and $\ep_i$, instead of $\gamma_i$.

\subsection{Products}
If $X, X'$ are affine geometric crystals, then so is $X \times X'$ \cite{BK, KNO}.  Let $(x,x') \in X \times X'$.  Then
\begin{equation}
\label{E:epphproduct}
\ep_k(x,x')= \frac{\ep_k(x) \ep_k(x')}{\ph_k(x')+\ep_k(x)} \qquad
\ph_k(x,x')=\frac{\ph_k(x) \ph_k(x')}{\ep_k(x)+ \ph_k(x')}
\end{equation}
and
\begin{align*}
e_k^c(x \otimes x') = (e_k^{c^+}x,e_k^{c/c^+}x')
\end{align*}
where
\begin{equation}\label{eq:c+}
c^+ = \frac{c\ph_k(x')+\ep_k(x)}{\ph_k(x') +\ep_k(x)}.
\end{equation}

\begin{remark}
Note that our notations differ from those in \cite{KNO} by swapping left and right in the product, and this agrees with \cite{LP2}. Note also that the birational formulae we use should be tropicalized using $(\min, +)$ (rather than $(\max,+)$) operations to yield the formulae for combinatorial crystals.
\end{remark}

\subsection{The basic geometric crystal}
We now introduce a geometric crystal $X_M$ which we call the {\it basic geometric crystal} of type $A_{n-1}^{(1)}$.  This is a geometric analogue of a limit of perfect crystals.  It is essentially the geometric crystal ${\mathcal B}_L(A_{n-1}^{(1)})$ of \cite[Section 5.2]{KNO}.

We have $X_M = \{x=(x^{(1)},x^{(2)},\ldots,x^{(n)}) \mid x^{(i)} \in \C^{*}\}$ and
$$
\ep_i(x) = x^{(i+1)} \qquad \ph_i(x) = x^{(i)} \qquad \gamma_i(x) = x^{(i)}/x^{(i+1)}
$$ 
and
$$
e_i^c: (x^{(1)},x^{(2)},\ldots,x^{(n)}) \longmapsto (x^{(1)},\ldots,cx^{(i)},c^{-1} x^{(i+1)}\ldots,x^{(n)}) .
$$
That $X_M$ is an affine geometric crystal is shown in \cite{KNO}.

\section{Geometric crystal structures on the unipotent loop group}
\subsection{Unipotent loop group}
Let $G= GL_n(\C((t)))$ denote the formal loop group, consisting of non-singular $n\times n$ matrices with complex formal Laurent series coefficients.  If $g =(g_{ij})_{i,j=1}^n \in G$ where $g_{ij}= \sum_k g_{ij}^k t^k$, we let $Y = Y(g)= (y_{rs})_{r,s\in \Z}$ denote the infinite periodic matrix defined by $y_{i+kn,j+k'n} = g_{ij}^{k'-k}$ for $1 \leq i,j\leq n$.  For the purposes of this paper we shall always think of formal loop group elements as infinite periodic matrices.

The unipotent loop group $U \subset G$ is defined as those $Y$ which are upper triangular, with 1's along the main diagonal.  We let $U^{\leq m} \subset U$ denote the subset consisting of infinite periodic matrices supported on the $m$ diagonals above the main diagonal.  Thus if $Y \in U$ then $Y \in U^{\leq m}$ if and only if $y_{ij} = 0$ for $j-i > m$.  It is clear that $U^{\leq m} \simeq \C^{mn}$.

\subsection{Whirls and Chevalley generators}\label{sec:whirlchev}
For $k \in \Z/n\Z$ and $a \in \C$, define the Chevalley generator $u_k(a) \in U$ by
$$
(u_k(a))_{ij} = \begin{cases} 
1& \mbox{if $j=i$}\\
a & \mbox{if $j=i+1=k+1$} \\
0& \mbox{otherwise.}
\end{cases}
$$
These are standard one-parameter subgroups of $U$ corresponding to the simple roots.

For $(x^{(1)},x^{(2)},\ldots,x^{(n)}) \in \C^n$, define the {\it {whirl}} $M(x) = M(x^{(1)},x^{(2)},\ldots,x^{(n)})$ \cite{LP} to be the infinite periodic matrix with 
$$
M(x)_{i,j} = 
\begin{cases}
1 & \text{if $j=i$;}\\
x^{(i)} & \text{if $j=i+1$;}\\
0 & \text{otherwise.}
\end{cases}
$$
Here and elsewhere the upper indices are to be taken modulo $n$.

\subsection{Geometric crystal structure}
Let $m \in \{1,2,\ldots\}$.  For $k \in \Z/n\Z$, define rational functions $\ep_k$ 
and $\ph_k$ on $U$ by 
$$
\ep_k(Y) = \frac{y_{k+1,k+m+1}}{y_{k+1,k+m}}  \qquad \ph_k(Y) = \frac{y_{k,k+m}}{y_{k+1,k+m}}.
$$
(These rational functions, and indeed the geometric crystal structure as well, extend to $G$.  However, we shall only consider the unipotent loop group.)

\begin{theorem}[Unipotent loop group geometric crystal for finite $m$]
\label{thm:Ucrystal}
Fix $m \in \{1,2,\ldots\}$.  The formal loop group $U$, the rational functions $\ep_k,\ph_k: U \to \C$, and the map 
\begin{equation}\label{E:action}
e_k^c: Y \longmapsto u_k((c-1) \ph_k)\; Y\; u_{k+m}((c^{-1}-1) \ep_k)
\end{equation}
form a geometric crystal.  Furthermore, the subvariety $U^{\leq m}$ is invariant under $e_k^c$, and thus is a geometric subcrystal.
\end{theorem} 

This result is not difficult to prove by direct calculation, but we omit the proof as it essentially follows from Corollary \ref{cor:same} (which shows that one has a geometric crystal structure on $U^{\leq m}$).

\begin{remark}
The formula \eqref{E:action} is essentially the same formula as that used by Berenstein and Kazhdan \cite[(2.14)]{BK2} to define a geometric crystal from a unipotent crystal.  Indeed, $U$ clearly has a $U \times U$ action (though $U^{\leq m}$ does not).  However, despite the simplicity of our construction, we have been unable to put it inside their framework of unipotent crystals.
\end{remark}

\subsection{Asymptotic geometric crystal structure}
Ley $Y \in U$.  Define the limit ratios \cite{LP}
$$
\ep_k(Y) = \lim_{m \to \infty} \frac{y_{k-m,k+1}}{y_{k-m,k}}  \qquad \ph_k(Y) = \lim_{m \to \infty} \frac{y_{k,k+m}}{y_{k+1,k+m}}
$$
assuming the limits exist.  
Since $\ep_k$ and $\ph_k$ are not rational functions, we cannot use them to construct an algebraic geometric crystal.  However, treating $\ep_k$ and $\ph_k$ as functions with a restricted domain, we can still formally construct a geometric crystal using the formula \eqref{E:action}.  

\begin{theorem}[Unipotent loop group geometric crystal for $m = \infty$] \label{thm:infinity}
The functions $\ep_k,\ph_k: U \to \C$, and the map 
$$
e_k^c: Y \longmapsto u_{k}((c-1) \ph_k)\; Y\; u_{k}((c^{-1}-1)\ep_k)
$$
satisfy the relations (2),(3),(4),(5) of a geometric crystal on $U$ (see Section \ref{sec:geom}). 
\end{theorem} 
Note that if $\ep_k$ and $\ph_k$ are defined at $Y \in U$, so is $e_k^c$, and in addition $e_k^c(Y)$ also
has this property.

\begin{proof}For simplicity we assume $n > 2$ and suppose that $\ph_k$ and $\ep_k$ are defined for $Y \in U$.  We have $\ph_j(Y u_k(a)) = \ph_j(Y)$ and $\ep_j(u_k(a) Y) = \ep_j(Y)$, and
$$
\ph_j(u_k(a) Y) =\begin{cases} \ph_k(Y) + a & \mbox{if $k = j$} \\ 
 \ph_{k-1}(Y)/(1 + a/\ph_k(Y)) &\mbox{if $k-1 = j$} \\
 \ph_j(Y) &\mbox{otherwise.}\end{cases}
$$
$$
\ep_j( Y u_{k}(a)) =\begin{cases} \ep_k(Y) + a & \mbox{if $k = j$} \\ 
 \ep_{k+1}(Y)/(1+ a/\ep_{k}(Y))& \mbox{if $k+1 = j$} \\
 \ep_j(Y) &\mbox{otherwise.} \end{cases}
$$
Relations (2),(3),(4) are immediate.  To obtain (5), one uses the relation $u_k(a)u_{k+1}(b)u_k(c) = u_{k+1}(bc/(a+c)) u_k(a+c) u_{k+1}(ab/(a+c))$ to get
\begin{align*}
&u_k\left((c'-1)\frac{c\phi_k}{1+(cc'-1)\phi_{k+1}/\phi_{k+1}}\right)
u_{k+1}((cc'-1)\phi_{k+1}) u_k((c-1)\phi_k)\\
&= u_{k+1}((c-1)c'\phi_{k+1})u_k((cc'-1)\frac{\phi_k}{1+(c'-1)\phi_k/\phi_k})u_{k+1}((c'-1)\phi_{k+1})
\end{align*}
where $\phi_k=\phi_k(Y)$.  A similar equality for $\ep_k$ gives (5).

The last statement of the theorem is straightforward.
\end{proof}

In place of axiom (1) of a geometric crystal, we may note that the domain of definition
of $\ep_k,\ph_k,e_k^c$ is dense in $U$ under (matrix-)entrywise convergence.  However,  $\ep_k$ and $\ph_k$ are not continuous in this topology.

\begin{remark}
The limit ratios $\ep_k$ and $\ph_k$ were introduced in \cite{LP} for the study of the totally nonnegative part $U_{\geq 0}$ of the loop group.  Indeed, these limits always exist for totally nonnegative elements which are not supported on finitely many diagonals.  They point to a deeper connection between total nonnegativity and crystals.
\end{remark}

\section{Product crystal structure and $S_m$-invariants}
\subsection{Birational $R$-matrix and invariants}
Let $X= X_1 \times X_2 \times \cdots \times X_m$ where each $X_i \simeq X_{M}$ is a basic affine geometric crystal.  We shall shift the indexing of the coordinates on $X_i$, so that
$$
X_i = \{x_i=(x_i^{(i)},x_i^{(i+1)},\ldots,x_i^{(i+n)})\}.
$$
That is, $\ph_1(x_i) = x_i^{(i)}$, and so on.

Define $$\k_{r}(x_j,x_{j+1}) = \sum_{s=0}^{n-1} (\prod_{t=1}^s x^{(r+t)}_{j+1} \prod_{t=s+1}^{n-1} x^{(r+t)}_j).$$  In the variables $x_i^{(r)}$, the birational $R$-matrix acts (see \cite[Proposition 3.1]{Y}) via algebra isomorphisms $s_1,s_2,\ldots,s_{m-1}$ of the field $\C(X)$ of rational functions in $\{x_i^{(r)}\}$, given by 
$$s_j(x^{(r)}_j) = \frac{x^{(r+1)}_{j+1}  \k_{r+1}(x_j,x_{j+1})}{\k_{r}(x_j,x_{j+1})} \text{\;\;\;\; and \;\;\;\;} s_j(x^{(r)}_{j+1}) = \frac{x^{(r-1)}_j \k_{r-1}(x_j,x_{j+1})}{\k_{r}(x_j,x_{j+1})}$$
and $s_j(x^{(r)}_k)=x^{(r)}_k$ for $k \neq j,j+1$.  The birational $R$-matrix acts as geometric crystal isomorphisms, and thus the crystal structure of $X$ descends to the invariants $\C(X)^{S_m}$.

For a sequence $x=(x_1,\ldots,x_m)$ of $n$-tuples of complex numbers, we define $M(x) \in U$ to be the product of whirls $M(x) = M(x^{(1)}_1,\ldots,x^{(n)}_1)\cdots M(x^{(1)}_m,\ldots,x^{(n)}_m)$.  The entries of $M(x)$ have the following description.  For $r \geq 1$ and $s \in \Z/n\Z$, define the {\it {loop elementary symmetric functions}}
$$
e_r^{(s)}(x_1,x_2,\ldots,x_m) = \sum_{1 \leq i_1 < i_2 < \cdots < i_r\leq m} x_{i_1}^{(s)} x_{i_2}^{(s+1)} \cdots x_{i_k}^{(s+k-1)}. $$
By convention we have $e_r^{(s)} = 1$ if $r = 0$, and $e_r^{(s)} = 0$ if $r < 0$.  Then by \cite[Lemma 7.5]{LP} we have
\begin{equation}
\label{E:Me}
M(x)_{i,j} = e^{(i)}_{j-i}(x_1,\ldots,x_m).
\end{equation}
Furthermore, it is shown in \cite[Section 6]{LP} that 
\begin{equation}\label{E:whirlcommute}
M(x_1)M(x_2) = M(s_1(x_1))M(s_1(x_2)).
\end{equation}

The ring $\LSym_m \subset \C[X]$ generated by the $e_r^{(s)}$ is called the ring of (whirl) loop symmetric functions. 

\begin{theorem} \label{thm:einv}
We have $\C(X)^{S_m} = \C(e_r^{(s)})$.  In particular, the $e_r^{(s)}$ are algebraically independent.  The map $X \to U^{\leq m}$ given by
$$
(x_1,\ldots,x_m) \mapsto M(x_1)M(x_2)\cdots M(x_m)
$$ 
identifies $\LSym_m = \C[e_r^{(s)}]$ with the coordinate ring $\C[U^{\leq m}]$.
\end{theorem}
\begin{proof}
It follows from \eqref{E:whirlcommute} that $e_r^{(s)} \in \C(X)^{S_m}$ (see also \cite{Y}).  Now consider the map $g: X \to \C^{mn}$ given by $x \mapsto (e_r^{(s)})^{s \in \Z/n\Z}_{r \in \{1,2,\ldots,m\}}$.  It is shown in \cite[Corollary 6.4, Proposition 8.2]{LP}  that when $x_i^{(k)}$ are positive real numbers and the products $R_i=\prod_{k\in \Z/n\Z} x_i^{(k)}$ are all distinct the map $g$ is $n!$ to $1$.  Since this is a Zariski-dense subset of $X$, we conclude that $[\C(X):\C(e_r^{(s)})] = m!$.  But then we must have $\C(X)^{S_m} = \C(e_r^{(s)})$.  Since the transcendence degree of $\C(X)^{S_m}$ is $mn$, it follows that the $e_r^{(s)}$ are algebraically independent.  The last statement follows immediately from \eqref{E:Me}.
\end{proof}

\begin{remark}
In a future work we shall show the stronger result that $\C(X)^{S_m}\cap \C[X] = \C[e_r^{(s)}]$.
\end{remark}

\subsection{Crystal structure in terms of invariants}

The aim of this section is to completely describe the geometric crystal structure of $X$ in terms of the invariants $e_r^{(s)}$.
\begin{thm}\label{lem:epph}
Let $x = (x_1,x_2,\ldots,x_m) \in X$.  We have
$$
\ep_k(x) = \frac{e_m^{(k+1)}}{e_{m-1}^{(k+1)}} \qquad \ph_k(x) = \frac{e_m^{(k)}}{e_{m-1}^{(k+1)}} \qquad \gamma_k(x) = \frac{e_m^{(k)}}{e_{m}^{(k+1)}}.
$$
\end{thm}
\begin{proof}
We proceed by induction on $m$.  For $m = 1$, we have $\ep_k(x) = x_1^{(k+1)}$ and $ \ph_k(x) = x_1^{(k)}$, agreeing with the theorem.

Let $x' = (x_1,\ldots,x_{m-1})\in X_1 \times X_2 \times \cdots \times X_{m-1}$.  Supposing the result is true for $m-1$, we calculate that
\begin{align*}
\ep_k(x',x_{m-1})
&= \frac{\ep_k(x') \ep_k(x_m)}{\ph_k(x_m)+\ep_k(x')} \\
&= \frac{e_{m-1}^{(k+1)}(x_1,\ldots,x_{m-1})x_m^{(k+m)}}{e_{m-2}^{(k+1)}(x_1,\ldots,x_{m-1})x_m^{(k+m-1)} + e_{m-1}^{(k+1)}(x_1,\ldots,x_{m-1})}\\
&=\frac{e_m^{(k+1)}(x_1,\ldots,x_m)}{e_{m-1}^{(k+1)}(x_1,\ldots,x_m)}.
\end{align*}
The calculation for $\ph_k$ is similar. For $\gamma_k$ we use $\gamma_k = \ph_k/\ep_k$.
\end{proof}

The next theorem completes the description of the geometric crystal structure in terms of the $e_r^{(s)}$. 

\begin{theorem}\label{thm:e}
The map $(e_k^c)^*:\C(X)^{S_m} \to \C(X)^{S_m}$ induced by $e_k^c: X \to X$ is given by $$M(x) \mapsto u_k((c-1) \ph_k)\; M(x)\; u_{k+m}((c^{-1}-1) \ep_k).$$ In other words  $e_r^{(s)}$ is sent to 
\begin{enumerate}
\item
$\qquad \qquad 
e_r^{(k)} + {(c-1)}  \ph_k e_{r-1}^{(k+1)} \qquad \text{if} \qquad s=k\neq k+m-r+1 \mod n
$
\item
$\qquad \qquad 
e_r^{(k+m-r+1)} + 
(c^{-1}-1) \; \ep_k  e_{r-1}^{(k+m-r+1)} \qquad \text{if} \qquad s= k+m-r+1\neq k \mod n
$
\item
$\qquad \qquad
e_r^{(k)}+
(c-1) \ph_k e_{r-1}^{(k+1)} + (c^{-1}-1) \ep_k e_{r-1}^{(k)} - \frac{(1-c)^2}{c} \ep_k \ph_k e_{r-2}^{(k+1)} 
$
if $s= k+m-r+1 =k \mod n$
\item
$\qquad \qquad
e_r^{(s)} \qquad \text {otherwise.}$
\end{enumerate}
\end{theorem}

\begin{example}
Let $m=n=2$. The action of $e_1^{c}$ on the crystal is induced by the transformation $M(x) \mapsto$
\begin{eqnarray*}
\left(
\begin{array}{cccccc}
\ddots & \vdots & \vdots & \vdots  & \vdots   \\
\dots& 1 & (c-1) \ph_1& 0& 0& \dots \\
\dots&0&1&0&0&\dots \\
\dots&0  & 0 & 1 & (c-1) \ph_1 &  \dots\\
\dots& 0 & 0 &0 & 1 & \dots \\
& \vdots & \vdots & \vdots  & \vdots    &
\ddots\end{array} \right)
M(x)
\left(
\begin{array}{cccccc}
\ddots & \vdots & \vdots & \vdots  & \vdots   \\
\dots& 1 & (1-c) \ep_1/c& 0& 0& \dots \\
\dots&0&1&0&0&\dots \\
\dots&0  & 0 & 1 & (1-c) \ep_1/c &  \dots\\
\dots& 0 & 0 &0 & 1 & \dots \\
& \vdots & \vdots & \vdots  & \vdots   &
\ddots\end{array} \right)
\end{eqnarray*}
where 
\begin{eqnarray*}
M(x) = \left(
\begin{array}{cccccc}
\ddots & \vdots & \vdots & \vdots  & \vdots   \\
\dots& 1 & e_1^{(1)}& e_2^{(1)}& 0& \dots \\
\dots&0&1& e_1^{(2)}&e_2^{(2)}&\dots \\
\dots&0  & 0 & 1 & e_1^{(1)} &  \dots\\
\dots& 0 & 0 &0 & 1 & \dots \\
& \vdots & \vdots & \vdots  & \vdots    &
\ddots\end{array} \right), \qquad \ph_1 = e_2^{(1)}/e_1^{(2)}, \qquad \ep_1 = e_2^{(2)}/e_1^{(2)}.
\end{eqnarray*}
\end{example}

\begin{proof}
The proof is by induction on $m$. For $m=1$ the statement is easily verified from the definition of basic geometric crystal. Assume now $m>1$.  We let $x = (x', x_m)$, where $x' = (x_1, \ldots , x_{m-1})$. Then 
$$M(x') \mapsto u_k((c^+-1) \ph_k(x'))\; M(x')\; u_{k+m-1}((c^+-1) \ep_k(x')),$$
$$M(x_m) \mapsto u_{k+m-1}((c/c^+-1) \ph_k(x_m))\; M(x_m)\; u_{k+m}((c^+/c-1)\ep_k(x_m)).$$
Plugging in 
$$c^+ = \frac{c \ph_k(x_m)+ \ep_k(x')}{\ph_k(x_m)+ \ep_k(x')}$$ and using \eqref{E:epphproduct} we obtain 
$$M(x')M(x_m) \mapsto u_k\left(\frac{(c-1) \ph_k(x_m)}{\ph_k(x_m)+ \ep_k(x')} \ph_k(x')\right) M(x')\; u_{k+m-1}\left(\frac{(1-c) \ph_k(x_m)}{c \ph_k(x_m)+ \ep_k(x')} \ep_k(x')\right) $$ $$u_{k+m-1}\left(\frac{(c-1) \ep_k(x')}{c \ph_k(x_m)+ \ep_k(x')} \ph_k(x_m)\right) M(x_m)\; u_{k+m}\left(\frac{(1-c) \ep_k(x')}{c (\ph_k(x_m)+ \ep_k(x'))} \ep_k(x_m)\right) = $$ $$u_k((c-1) \ph_k(x))\; M(x)\; u_{k+m}((c^{-1}-1) \ep_k(x)).$$
\end{proof}

If $X$ is a geometric crystal and $\Gamma$ is a group of crystal automorphisms of $X$, we say that a geometric crystal $X'$ is the quotient of $X$ by $\Gamma$, if we have a rational map $X \to Y$, commuting with the geometric crystal structure, and inducing $\C(Y) \simeq \C(X)^{\Gamma}$.

\begin{cor}\label{cor:same}
The map $(x_1,\ldots,x_m) \mapsto M(x_1)\cdots M(x_m)$ identifies the quotient of the product geometric crystal $X = X_M^m$ by the $S_m$ birational $R$-matrix action with the geometric crystal on $U^{\leq m}$ of Theorem \ref{thm:Ucrystal}.
\end{cor}

\begin{remark}
Our whirls play a similar role to the $M$-matrices of \cite{KNO}, which have been studied for all affine types.  This suggests that Theorem \ref{thm:e} and also our work on total positivity \cite{LP} may have a generalization to other types in the same spirit.
\end{remark}


\section{Action of crystal operators on the energy function}

We recall the definition of the {\it {loop Schur functions}} that were introduced in \cite{LP}. Let $\lambda/\mu$ be a skew Young diagram shape. A square $s = (i,j)$ in the $i$-th row and $j$-th column has {\it content} $c(s)=i-j$.  We caution that our notion of content is the negative of the usual one.  Recall that a semistandard Young tableaux $T$ with shape
$\lambda/\mu$ is a filling of each square $s \in \lambda/\mu$ with an
integer $T(s) \in \Z_{> 0}$ so that the rows are weakly-increasing, and columns are increasing.  For $r\in \Z/n\Z$, the $r$-weight $x^T$ of a tableaux $T$ is given by $x^T = \prod_{s \in \lambda/\mu}x_{T(s)}^{(c(s)+r)}$.

We shall draw our shapes and tableaux in English notation:
$$
\tableau[sY]{\bl \times&\bl \times&\bl \times&\bl \times&\bl \times& \circ &&& \bullet \\\bl \times&\bl \times&\bl \times& \circ &&&& \bullet \\\bl \times&\bl \times&\bl \times&&& \bullet} \qquad
\tableau[sY]{\bl &\bl&1&1&1&3\\1&2&2&3&4\\3&3&4}
$$
For $n = 3$ the $0$-weight of the above tableau is $(x_1^{(1)})^2 (x_3^{(1)})^3
x_1^{(2)} x_2^{(2)} x_3^{(2)} x_1^{(3)} x_2^{(3)} (x_4^{(3)})^2.$
We define the {\it loop (skew) Schur function} by
$$
s^{(r)}_{\lambda/\mu}({x}) = \sum_{T} x^T
$$
where the summation is over all semistandard Young tableaux of
(skew) shape $\lambda/\mu$.  Loop Schur functions are of significance in the theory of geometric crystals because of the following result.  Let $\delta_m = (m,m-1,\ldots,1)$ denote the staircase shape of size $m$.

\begin{theorem} \cite[Theorem 1.2]{LP2}
The birational energy function $\D_B$ of $X = X_1 \times \cdots \times X_m$ is the loop Schur function $s^{(0)}_{(n-1)\delta_{m-1}}(x_1,\ldots,x_m)$.
\end{theorem} 

Note that unlike $\ep_k$ and $\ph_k$, the birational energy function is a polynomial.  We have the following analog of the Jacobi-Trudi formula.

\begin{thm} \cite[Theorem 7.6]{LP} \label{thm:JT}
We have $s^{(r)}_{\lambda/\mu} = \det(e_{\lambda'_i-\mu'_j-i+j}^{(r-j+1+\mu'_j)})$.
\end{thm}

This means that the skew loop Schur functions $s^{(r)}_{\lambda/\mu}(x_1,x_2,\ldots,x_m)$ are minors of the infinite periodic matrix $M(x) = M(x_1)\cdots M(x_m)$.  

For a skew shape $\lambda/\mu$ we distinguish its NW corners and SE corners (in English convention). The figure above shows the NW corners marked with $\circ$ and the SE corners marked with $\bullet$. Let $A(\lambda/\mu)_k^{(r)}$ (resp. $B(\lambda/\mu)_k^{(r)}$) denote the set of NW (resp. SE) corners of $\lambda/\mu$ of color $k$, that is corner cells $(i,j)$ with $i-j+r=k \mod n$.  
The proof of the following result follows from an examination of the effect of row and column operations  on Jacobi-Trudi matrices, and is omitted.

\begin{theorem}\label{thm:schuraction}
 We have \begin{eqnarray*} (e_k^c)^*(s^{(r)}_{\lambda/\mu}) = \sum_{B \subset B(\lambda/\mu)_{k+m}^{(r)}\; A \subset A(\lambda/\mu)_{k}^{(r)} \; A \cap B = \emptyset} 
((c^{-1}-1)\ep_k)^{|B|}\, ((c-1)\ph_k)^{|A|}\, s^{(r)}_{\lambda/\mu - A - B}.\end{eqnarray*}
\end{theorem}

We recover the following property of the birational energy function $\D_B$, paralel to the analogous property of the combinatorial energy function. 

\begin{corollary}
 For $k \not = 0 \mod n$ we have $(e_k^c)^*(\D_B) = \D_B$.
\end{corollary}

\begin{proof}
The shape $(n-1) \delta_{m-1}$ has a single NW corner of color $0$ and all its SE corners are of color $m$. Thus only for $k=0$ the sets $A((n-1) \delta_{m-1})_k^{(0)}$ and $B((n-1) \delta_{m-1})_{k+m}^{(0)}$ are non-empty.
\end{proof}

\end{document}